\documentclass[11pt, a4paper, reqno, english]{smfart}

\usepackage{amsfonts, amsthm, amsmath, amscd, amssymb}

\renewcommand\AA{\mathbb{A}}
\newcommand\QQ{\mathbb{Q}}

\newcommand\NN{\mathbb{N}}
\newcommand\RR{\mathbb{R}}

\newcommand\ZZ{\mathbb{Z}}

\newcommand\PP{\mathbb{P}}
\newcommand\x{\mathbf{x}}

\DeclareMathOperator{\Diag}{Diag}
\DeclareMathOperator{\Pic}{Pic}
\DeclareMathOperator{\rank}{rank}

\renewcommand{\leq}{\leqslant}
\renewcommand{\geq}{\geqslant}

\newcommand\ve{\varepsilon}

\renewcommand\phi{\varphi}

\newcommand\al{\alpha}
\newcommand\be{\beta}

\newcommand{\ring}{\mathfrak{o}_K}
\newcommand{\fa}{\mathfrak{a}}
\newcommand{\fn}{\mathfrak{n}}
\newcommand{\fp}{\mathfrak{p}}

\newtheorem*{thm*}{Theorem}
\newtheorem{lemma}{Lemma}

\theoremstyle{definition}
\newtheorem*{ack}{Acknowledgements}

\newcommand{\A}{\mathcal{A}}

\newcommand\cQ{{\overline{\mathbb{Q}}}}

\begin{document}

\title{Linear growth for Ch\^atelet surfaces}

\author{T.D. Browning}
\address{School of Mathematics\\
University of Bristol\\ Bristol\\ BS8 1TW}
\email{t.d.browning@bristol.ac.uk}

\date{\today}

\begin{abstract}
An upper bound of the expected order of magnitude is established for
the number of $\QQ$-rational points of bounded height on Ch\^atelet
surfaces defined over $\QQ$.
\end{abstract}

\subjclass{11D45 (14G05)}

\maketitle

\section{Introduction}

A Ch\^{a}telet surface $X$ over $\QQ$ is a proper smooth model of 
an affine surface in $\AA^3$ of the form
\begin{equation}\label{eq:chat}
y^2-az^2=f(x), 
\end{equation}
where $a\in \ZZ$ is not a square and $f\in \ZZ[x]$ is a
polynomial without repeated roots and degree 
$3$ or $4$.  In the birational classification of rational surfaces
summarised by Iskovskikh \cite{isk},  
Ch\^atelet surfaces appear as some of the simplest non-trivial surfaces. They 
are conic bundle surfaces of degree $4$, being
equipped with a dominant morphism 
$$
\pi:X \rightarrow \PP^1,
$$ 
all
of whose geometric fibres are conics.  If  $-K_X$ denotes the 
anticanonical divisor, then the linear
system $|-K_X|$ has no base point and gives a morphism
$\psi:X\rightarrow \PP^4$ whose image 
is a singular del Pezzo surface of degree $4$.

Writing $H=H_4\circ
\psi$, where $H_4:\PP^4(\QQ)\rightarrow \RR_{>0}$ is the exponential
height metrized by an arbitrary choice of norm, the primary goal of
this paper is to study the asymptotic behaviour of 
$$
N(B)=\#\{x \in X(\QQ): H(x)\leq B\},
$$
as $B\rightarrow \infty$. We will assume that $X(\QQ)\neq \emptyset$
for all of the Ch\^atelet surfaces under consideration here. The
problem of determining when $X(\QQ)\neq \emptyset$
is completely handled by the 
work of Colliot-Th\'el\`ene, Sansuc and Swinnerton-Dyer \cite{ct1,ct2}.
Our investigation of the counting function $N(B)$ is guided by a well-known conjecture of Manin
\cite{f-m-t},  which predicts the existence of
a constant $c_X> 0$ such that
$
N(B)\sim c_X B (\log B)^{\rho_X-1},
$
as $B\rightarrow \infty$, 
where $\rho_X$ is the rank of the Picard group of $X$.
  With this in mind, the following is our main result.

\begin{thm*}
Let $X$ be a Ch\^{a}telet surface defined over $\QQ$, arising as a proper smooth model of 
the affine surface \eqref{eq:chat}.  Assume that $a<0.$
Then we have 
$$
N(B)=O(B(\log B)^{\rho_X-1}),
$$
where $\rho_X$ is the rank of the Picard group of $X$.
\end{thm*}

Here, as throughout our work, the implied constant is allowed to
depend upon the surface.
Although we will not present any details, it transpires that similar,
but more intricate, arguments also permit one to handle the case $a>0$
in the theorem.

Let 
$$
\be_X=\lim_{B\rightarrow \infty} \frac{\log N(B)}{\log B}
$$ 
be the
growth rate of $X(\QQ)$. As a crude corollary of our theorem it
follows that $\be_X\leq 1$ for Ch\^atelet surfaces.
The question of obtaining lower bounds has recently been addressed by Iwaniec and Munshi
\cite{munshi}, with an analysis of the case in which  $f$ is taken to be an irreducible cubic
polynomial in \eqref{eq:chat}.  Their lower bound is difficult to compare with our work,
however, since they work with a different height function. 
In forthcoming work of la Bret\`eche, Browning and 
Peyre, a resolution of the Manin
conjecture is achieved for a family of Ch\^atelet surfaces that corresponds to taking
$a=-1$ and $f$ a polynomial that is totally reducible into linear factors over $\QQ$. 

According to the investigation of Iskovskikh \cite[Proposition~1]{isk} a conic
bundle surface $X/\PP^1$ of degree $4$ arises in two possible ways.
Either the anticanonical divisor $-K_X$ is not 
ample, in which case $X$ is 
a Ch\^atelet
surface, or else $-K_X$ is ample, in which case $X$ is a non-singular
quartic del Pezzo surface.
Our proof of the theorem makes essential use of the conic bundle
structure of Ch\^atelet surfaces. 
It is inspired by an approach adopted by Salberger, as 
communicated  at the conference ``Higher-dimensional varieties and rational
points'' in Budapest in 2001, for the class of non-singular 
quartic del Pezzo surfaces with a conic bundle structure.
For such surfaces an upper bound  
$O_{\ve}(B^{1+\ve})$ is achieved for the corresponding counting
function by taking advantage of the morphism $\pi:X\rightarrow \PP^1$ in order
to count rational points of bounded height on  
the conics $\pi^{-1}(p)$, uniformly for points $p \in \PP^1(\QQ)$ of small
height.  In subsequent work 
Leung \cite{fok} has refined this argument, replacing $B^\ve$ by $(\log B)^{A}$ for a certain integer
$A\leq 5$. However, the value of $A$ is often bigger than the exponent
predicted by Manin. 
A pedestrian translation of these arguments from del Pezzo surfaces to 
Ch\^atelet surfaces would lead to a similar deficiency. 
To overcome this, we will gain significant extra leverage by restricting the
summation to only those $p\in \PP^1(\QQ)$ of small height that produce isotropic conics
$\pi^{-1}(p)$.  It seems likely that this innovation could also be put
to use in the analogous situation studied by Leung \cite{fok}.

\begin{ack}
This article addresses a question that was posed by J.-L.
Colliot-Th\'el\`ene at the meeting 
``Rational points on curves and higher-dimensional varieties'' in
Warwick in June 2008. It is a pleasure to thank 
R. de la Bret\`eche, A. Gorodnik and O. Wittenberg
for a number of useful comments, in addition to the anonymous referee
for his careful reading of the manuscript.
While working on this paper the author was supported by EPSRC
grant number \texttt{EP/E053262/1}.
\end{ack}

\section{Geometric preliminaries}

Let $F(u,v)=v^4f(\frac{u}{v})$, a binary quartic form with integer
coefficients. We denote by $X_1\subset \PP^2\times \AA^1$
the hypersurface
$$
y_1^2-az_1^2=t_1^2F(u,1), 
$$
and by  $X_2\subset \PP^2\times \AA^1$
the hypersurface
$$
y_2^2-az_2^2=t_2^2F(1,v).
$$
The Ch\^atelet surface associated with \eqref{eq:chat} is the 
geometrically integral smooth projective surface obtained by patching
together $X_1,X_2$ via the isomorphism
\begin{align*}
X_1\setminus \{u=0\} &\longrightarrow X_2\setminus \{v=0\},\\ 
([y_1,z_1,t_1];u)&\longmapsto ([y_1,z_1,u^2t_1];u^{-1}).
\end{align*}

Since $f$ has non-zero discriminant, we have a factorisation
\[ 
F(u,v)=(\be_1u-\al_1v)(\be_2u-\al_2v)(\be_3u-\al_3v)(\be_4u-\al_4v), 
\]
over $\cQ$, with $[\al_1,\be_1],\ldots,[\al_4,\be_4]\in \PP^1(\cQ)$ distinct.
The morphisms  $X_1 \rightarrow \PP^1$ and $X_2 \rightarrow \PP^1$ 
given by $([y_1 , z_1 , t_1 ]; u )\mapsto [u , 1]$ and 
$([y_2 , z_2 , t_2]; v )\mapsto [1 , v ]$, respectively, 
glue together to give a conic fibration $\pi: X \rightarrow \PP^1$. It
has four degenerate geometric fibres over the points 
$p_i = [\al_i , \be_i ] \in \PP^1(\cQ)$, for $1 \leq i\leq 4$.
The geometric fibre above $p_i$ is the subvariety of $X$ defined by $u=\al_i$
and $y_1\pm \sqrt{a}z_1=0$. This defines a union of two geometrically
integral divisors that intersect transversally and are both isomorphic
to $\PP^1$ over $\cQ$.

Let $\Pic(X)$ be the Picard group of $X$. Then 
$\Pic(X)$ is a torsion-free $\ZZ$-module with finite rank $\rho_X$, say. 
An explicit description of $\rho_X$ is given in the following result.

\begin{lemma}\label{lem:pic}
Suppose that $f=f_1\cdots f_r$ is the factorisation into irreducibles of $f$ over
$\QQ$. For each $1\leq i\leq r$ let $\QQ_{f_i}=\QQ[x]/(f_i)$ denote the 
field obtained by adjoining a root of $f_i$ to $\QQ$. Then we have
$$
\rho_X= 2+\#\{1\leq i\leq r: \sqrt{a}\in \QQ_{f_i}\}.
$$
\end{lemma}

\begin{proof}
There is a homomorphism $\Pic (X)\rightarrow \ZZ$, 
which to a divisor class in $\Pic( X)$ associates its intersection 
number with the fibre of the morphism
$\pi: X\rightarrow \PP^1$ above a closed point of $\PP^1$.  
The image of this map has finite index in $\ZZ$. 
Moreover, the kernel is generated by the
``vertical'' divisors, 
namely those which are supported in finitely many fibres of $\pi$.

We now choose an irreducible fibre of the morphism
$\pi: X\rightarrow \PP^1$. Furthermore, in each reducible fibre, we choose one of the
two components. Let $D$ be the free abelian group generated by all
of these divisors.  It plainly follows that 
$$
\rank(D)= 1+\#\{1\leq i\leq r: \sqrt{a}\in \QQ_{f_i}\},
$$
since the residue field of the closed point
corresponding to $f_i$ is just $\QQ_{f_i}$.

Finally, we note that the natural map from $D$ to $\Pic(X)$ is injective and
it identifies $D$ with the kernel of $\Pic(X)\rightarrow \ZZ$. 
Therefore we have 
$$
\rho_X=1+\rank(D),
$$
as required to complete the proof of the lemma. 
%
\end{proof}

\section{Proof of the theorem}

In what follows it will be convenient to use the notation $Z^m$ for
the set of primitive vectors in $\ZZ^m$. The following result
translates the problem to one involving a family of conics. 

\begin{lemma}\label{lem:torsor}
Suppose that the exponential height  $H_4$ on $\PP^4(\QQ)$ is metrized
by a norm $\|\cdot\|$ on $\RR^5$. Then we have $N(B)= \frac{1}{4}T(B)$, where
$$
T(B)=\#\Big\{(y,z,t;u,v)\in Z^3\times Z^2: 
\begin{array}{l}
\|(v^2t,uvt,u^2t,y,z)\|\leq B\\
y^2-az^2=t^2 F(u,v)
\end{array}
\Big\}.
$$
\end{lemma}

\begin{proof}
Suppose that $f(x)=c_0x^4+\cdots +c_4$ in \eqref{eq:chat} for $c_i\in
\ZZ$. 
Consider the maps $\psi_i: X_i\rightarrow \PP^4$ given by 
\begin{align*}
&\psi_1:
([y_1,z_1,t_1];u)\longmapsto [t_1, ut_1, u^2t_1,y_1,z_1],\\
&\psi_2:
([y_2,z_2,t_2];v)\longmapsto [v^2t_2, vt_2, t_2,y_2,z_2].
\end{align*}
These induce a morphism $\psi:X\rightarrow \PP^4$ whose image 
is the del Pezzo surface 
$$
\begin{cases}
  x_0x_2=x_1^2,\\
  x_3^2-ax_4^2=c_4x_0^2+c_3x_0x_1+c_2x_0x_2+c_1x_1x_2+c_0x_2^2,
\end{cases}
$$
which we denote by $Y$. Let us write $Q(x_0,x_1,x_2)$ for the
quadratic form appearing on the right hand side of the second
equation. 

Let $H=H_4\circ \psi$, where $H_4$ is the exponential
height on $\PP^4(\QQ)$ defined by $H_4([\x])=\|\x\|$ if $\x \in Z^5$. 
Then we have
$$
N(B)=\frac{1}{2}\#\{\x \in Z^5: [\x]\in Y, ~\|\x\|\leq B\}.
$$
There is a $1:2$ correspondence between integer solutions of the
equation $x_0x_2=x_1^2$ and vectors $(t,u,v)\in\ZZ^3$ such that $u,v$
are coprime, given by 
$(x_0,x_1,x_2)=t(v^2,uv,u^2)$.
Furthermore, the primitivity of $\x$ is equivalent to the vector
$(t,x_3,x_4)$ being primitive. 
Substituting this into the second equation and 
noting that $Q(v^2,uv,u^2)=F(u,v)$, we 
therefore arrive at the statement of Lemma~\ref{lem:torsor}.
\end{proof}

In our work we are only interested in an upper bound for $N(B)$. By
equivalence of norms it will suffice to work with the norm
$\|\x\|=\max_{0\leq i \leq 4}|x_i|$ on $\RR^5$. 
Since $a$ is not a square we must have $|t|\geq 1$ 
in each solution to be counted, whence $\max\{u^2,v^2\}\leq B$. 
There will be no loss of generality in fixing attention on the
contribution from $u,v$ such that $|u|\leq |v|$. 
Let $\A$ denote the set of $(u,v)\in Z^2$ for which 
$|u|\leq |v|\leq \sqrt{B}$ and $F(u,v)\neq 0$.  Then it follows from
Lemma \ref{lem:torsor} that
$$
N(B)\ll \sum_{(u,v)\in \A}
M_{u,v}(B), 
$$
where
$$
M_{u,v}(B)=
\#\Big\{(y,z,t)\in Z^3: 
\begin{array}{l}
\max\{v^2|t|,|y|,|z|\}\leq B\\
y^2-az^2=t^2 F(u,v)
\end{array}
\Big\}.
$$
We would now like to thin down the outer summation
by restricting attention to those $(u,v)\in \A$ for which the conic 
$y^2-az^2=t^2 F(u,v)$ has a non-trivial rational point.

For our purposes it will suffice to restrict attention to those $(u,v)\in \A$
for which the Legendre symbol $(\frac{a}{p})$ is distinct from $-1$ for 
each odd prime $p$ such that $p\| F(u,v)$.
Here we write $p\| n$ for $n \in \ZZ$ if $p\mid n$ but $p^{2}\nmid n$.
To see that this is satisfactory one merely notes that if 
$p\| F(u,v)$ then the equation for the conic implies that 
$y^2\equiv a z^2 \bmod{p}$ and $p\nmid
\gcd(y,z)$, since $\gcd(y,z,t)=1$ in each solution counted.

Define the arithmetic function
\begin{equation}
  \label{eq:phi}
\vartheta(n)=\prod_{p\| n} 2^{-1}\Big(1+\Big(\frac{a}{p}\Big)\Big),
\end{equation}
where we have extended the Legendre symbol to all primes by setting
$(\frac{a}{2})=0$. 
The function $\vartheta$ is multiplicative, non-negative  and satisfies 
$$
\vartheta(p^\ell)=\begin{cases}
\frac{1}{2}, & \mbox{if $\ell=1$ and $p\mid 2a$,}\\
0, & \mbox{if $\ell=1$ and $(\frac{a}{p})=-1$,}\\
1, & \mbox{otherwise,}
\end{cases}
$$
for any prime power $p^\ell$.
We will use $\vartheta$ as a characteristic function to  weed out
values of $(u,v)\in \A$ that produce anisotropic 
conics. In this way we obtain
$$
N(B)\ll \sum_{(u,v)\in \A}\vartheta(|F(u,v)|)M_{u,v}(B).
$$

The task of estimating $M_{u,v}(B)$ boils down to counting rational points
on a geometrically integral plane conic, with the points constrained to lie in a
lop-sided region. For this we can take advantage 
of work of Browning and Heath-Brown \cite[Corollary~2]{bhb}, a
key feature of which being its uniformity with respect to the
height of the conic. Since $a<0$ it follows
that we may replace the
height restrictions on $y,z,t$ in $M_{u,v}(B)$ by
$$
y,z\ll \frac{B|F(u,v)|^{\frac{1}{2}}}{v^2}, \quad |t|\leq \frac{B}{v^2},
$$
as follows from the equation for the conic. 
Our conic is defined by a ternary quadratic form $\x^T \mathbf{M}\x$,
where $\x=(y,z,t)$ and $\mathbf{M}=\Diag(1,-a,-F(u,v))$.
In particular the greatest common divisor of the $2\times 2$ minors of
$\mathbf{M}$ is $O(1)$ and its determinant is $aF(u,v)$. 
The inequalities satisfied by $y,z,t$ above define a box in $\RR^3$
with volume $O(v^{-6}B^3|F(u,v)|)$. It now follows from 
\cite[Corollary~2]{bhb} that 
$$
M_{u,v}(B)\ll 2^{\omega(F(u,v))}\Big(1+\frac{B}{v^2}\Big)\ll 
B\frac{2^{\omega(F(u,v))}}{v^2},
$$ 
since $|v|\leq \sqrt{B}$.
Note that we have replaced the divisor function by the function 
$2^{\omega(\cdot)}$ in our application of this result, where $\omega(n)$ denotes the
number of distinct prime divisors of $n$. An inspection of the proof
reveals that is indeed permissible. 

Let $\varpi(n)=2^{\omega(n)}\vartheta(n)$, where $\vartheta$ is given by
\eqref{eq:phi}. Then our analysis so far has shown that 
\begin{equation}
  \label{eq:return}
N(B)\ll B\sum_{(u,v)\in \A}\frac{\varpi(|F(u,v)|)}{v^2}.
\end{equation}
In view of the trivial bound $\varpi(n)=O_\ve(n^{\ve})$ for any
$\ve>0$, it would be easy to conclude at this point that
$N(B)=O_{\ve}(B^{1+\ve})$. To get the correct power of $\log B$
emerging we must work somewhat harder.
For given $U,V\geq 1$ it will be convenient to introduce the sum
$$
S(U,V)=\sum_{|u|\leq U}\sum_{|v|\leq V} \varpi(|F(u,v)|).
$$
The estimation of $S(U,V)$ is the subject of the following result,
whose proof we will defer to the next section.

\begin{lemma}\label{lem:tech}
Let $V\geq U\geq 1$. Then for all $\ve>0$ we have 
$$
S(U,V)\ll_{\ve} UV (\log V)^{\rho_X-2}+V^{1+\ve}.
$$
\end{lemma}

We now have everything in place to complete the proof of the
theorem. Returning to \eqref{eq:return} we see that there is an
overall contribution of $O(B)$ from those $(u,v)\in\A$ with $u=0$. 
Breaking the summation of the remaining 
$u,v$ into dyadic intervals we therefore find
\begin{align*}
N(B)&\ll B+B
\sum_{\substack{i,j\in \ZZ\\
-1\leq i<j\leq \frac{1}{2\log 2}\log B}}
\sum_{2^i<|u|\leq 2^{i+1}}
\sum_{2^j<|v|\leq 2^{j+1}}
\frac{\varpi(|F(u,v)|)}{v^2}\\
&\ll B+B
\sum_{\substack{i,j\in \ZZ\\
-1\leq i<j\leq \frac{1}{2\log 2}\log B}}
\frac{S(2^{i+1},2^{j+1})}{2^{2j}}.
\end{align*}
Here we have 
dropped the conditions that $F(u,v)\neq 0$ and $\gcd(u,v)=1$, as
permitted by the fact that the summand is non-negative. 
Applying Lemma~\ref{lem:tech} we therefore deduce that
\begin{align*}
N(B)\ll B+B(\log B)^{\rho_X-2}
\sum_{\substack{i,j\in \ZZ\\
-1\leq i<j\leq \frac{1}{2\log 2}\log B}}
\frac{2^{i}}{2^{j}}\ll B(\log B)^{\rho_X-1},
\end{align*}
as required to complete the proof
of the theorem.

\section{Proof of Lemma \ref{lem:tech}}

Determining the average order of arithmetic functions
as they range over the values of polynomials has a substantial
pedigree in analytic number theory. For the proof of Lemma
\ref{lem:tech} we will need to analyse the average order of the
arithmetic function
$$
\varpi(n)=2^{\omega(n)} \prod_{p\| n}
2^{-1}\Big(1+\Big(\frac{a}{p}\Big)\Big),
$$
as it ranges over the values of the binary quartic form $F$.

The key technical tool for this argument 
is supplied by  work of la Bret\`eche and Browning \cite{nair}. We
recall that $F$ has non-zero discriminant and observe that $\varpi$ is a non-negative
multiplicative arithmetic function satisfying the estimates
$\varpi(n)=O_\ve(n^\ve)$ and $\varpi(p^\ell)\leq 2$ for $n\in \NN$ and
prime powers $p^\ell$.
In view of the fact that $\varpi(p)=1+(\frac{a}{p})$, we may therefore
conclude from \cite[Corollary 1]{nair} that 
$$
S(U,V) \ll_{\ve}
UVE_f(U) +V^{1+\ve},
$$
for any $\ve>0$, where 
$$
E_f(U)=
\prod_{1\ll p\leq U}\Big(1+\frac{\rho_f(p)(\frac{a}{p})}{p}\Big).
$$
Here $f(x)=F(x,1)$ is the polynomial appearing in
\eqref{eq:chat} and $\rho_{f}(m)$ is 
the number of solutions to the congruence 
$f(x)\equiv 0 \bmod{m}$ in $\ZZ/m\ZZ$.

Suppose that $f=f_1\cdots f_r$ is the factorisation into irreducibles of $f$ over
$\QQ$, with 
$\sum_{i=1}^r\deg f_i\in \{3,4\}$. Then we have 
$$
E_f(U)\ll E_{f_1}(U)\cdots E_{f_r}(U).
$$
Our attention now shifts to estimating $E_f(U)$ for any $U\geq 2$ and
any irreducible polynomial $f \in \ZZ[x]$ of degree $d$ with non-zero discriminant.
We will show that
\begin{equation}
  \label{eq:show}
  E_f(U)\ll \begin{cases}
1, & \mbox{if $\sqrt{a}\not \in \QQ_{f}$,}\\
\log U, & \mbox{if $\sqrt{a}\in \QQ_{f}$,}
\end{cases}
\end{equation}
where 
$\QQ_{f}=\QQ[x]/(f)$ denotes the 
field obtained by adjoining a root of $f$ to $\QQ$.
In view of Lemma \ref{lem:pic} this will 
suffice for the statement of Lemma \ref{lem:tech}.

In order to understand the asymptotic behaviour of $E_f(U)$ we must
investigate the analytic properties of the $L$-function
$$
H(s)=\prod_p\Big(1+ \frac{\rho_f(p)(\frac{a}{p})}{p^s}\Big), \quad (\Re(s)>1).
$$
We will do so by relating $H(s)$ to a certain Hecke $L$-function and
analysing it in the neighbourhood of $s=1$.  The
necessary facts are classical and can be found in the work of Heilbronn
\cite{heilbronn} and Neukirch \cite[Chapter VII]{neukirch}.

Let $K=\QQ_f$ and 
and let $d_K$ denote its discriminant.
We write $\fa=(2ad_K)$ for the ideal generated by $2ad_K$ in $\ring$. 
Furthermore, let $N(\fn)=|\ring/\fn|$ denote the norm of any ideal $\fn$ in
$\ring$.
For a prime ideal $\fp$ in $\ring$ which is coprime to $\fa$ we define 
$$
\chi(\fp)=\Big(\frac{a}{N(\fp)}  \Big)=\Big(\frac{a}{p}  \Big)^\ell,
$$
if $N(\fp)=p^\ell$, where $(\frac{a}{p})$ is the ordinary Legendre
symbol. One extends $\chi$ to all fractional ideals coprime to $\fa$
by multiplicativity. 
Then $\chi$ is a group homomorphism from the ray class group
 $J^{\fa}/P^{\fa}$ to $\{\pm 1\}$. Here $J^{\fa}$ is the group of ideals
coprime to $\fa$ and $P^{\fa}$ is the subgroup of fractional principal
 ideals $(\alpha)$ for which $\al \equiv 1 \bmod{\fa}$ and
 $\sigma(\al)>0$ for every real embedding $\sigma:K\rightarrow
 \RR$. Thus $\chi$ is a generalised Dirichlet character modulo $\fa$. 
A principal character modulo $\fa$ is any character $\chi_0$ such
that $\chi_0(\fn)=1$ for all $\fn\in J^{\fa}$.
Finally, we extend $\chi$ to all integral ideals by setting $\chi(\fn)=0$ if
$\fn$ has a factor in common with $\fa$.

The Hecke $L$-function associated to the number field $K$ and the
quadratic character $\chi$ is defined to be
$$
L_K(s,\chi)=
\sum_{\fn}\frac{\chi(\fn)}{N(\fn)^s}=
\sum_{n=1}^\infty \frac{b(n)}{n^s}, \quad
(\Re(s)>1),
$$
where the first sum is over integral ideals and 
$b(n)=\sum_{N(\fn)=n}\chi(\fn)$. 
For a rational prime $p$ one notes
that 
$$
b(p)=\sum_{N(\fp)=p}\chi(\fp)=\Big(\frac{a}{p}\Big)\#\{\fp : N(\fp)=p\}.
$$
Now for $p\nmid d_K$ there is a well-known principle due to
Dedekind \cite[p. 212]{dedekind} 
which ensures that $\rho_f(p)=\#\{\fp : N(\fp)=p\}$.
A modern account of this fact can be found in the work of Narkiewicz
\cite[\S 4.3]{nark}.
Employing a standard calculation based on partial summation and the prime ideal
theorem, we therefore conclude that 
$$
E_f(U)\ll \exp\Big(\sum_{N(\fp)\leq U}\frac{\chi(\fp)}{N(\fp)}\Big)
\ll
\begin{cases}
1, & \mbox{if $\chi\neq \chi_0$,}\\
\log U, & \mbox{if $\chi=\chi_0$.}
\end{cases}
$$

In order to complete the proof of \eqref{eq:show} it therefore remains
to show that $\chi$ is principal if and only if $\sqrt{a}\in K$. 
For any finite extension $N/M$ of number fields, let $P(N/M)$ denote
the set of all unramified prime ideals of $M$ which admit in $N$ a
prime divisor of residue class degree $1$ over $M$. 

Suppose first that $\sqrt{a}\in K$ and write $J=\QQ(\sqrt{a})$. 
Then we have a tower of separable extensions $\QQ\subseteq J \subseteq K$.
If $\chi$ is not principal then there is a prime ideal $\fp$ above a
rational prime $p\nmid 2ad_K$ with odd residue class degree, such that
$(\frac{a}{p})=-1$.  In particular $p$ is inert in $J$, with residue
class degree $2$.  But then the transitivity of norms implies that the
residue class degree of $\fp$ is even, which is a 
contradiction.

We now argue in the reverse direction, taking for our hypothesis the
assumption that $\chi$ is principal. This is equivalent to
$\chi(\fp)=1$ for all prime ideals $\fp$ coprime to $\mathfrak{a}.$
Any unramified rational prime $p$ factorises as 
$(p)=\fp_1\cdots \fp_r$ with distinct primes $\fp_i$ such that 
$N(\fp_i)=p^{\ell_i}$  and $\sum_{i=1}^r\ell_i=[K:\QQ].$
It follows that $(\frac{a}{p})=1$  for any prime $p\nmid 2ad_K$ such
that $(p) \in P(K/\QQ)$. But then any such $p$ splits
completely in  $J=\QQ(\sqrt{a})$ since $(\frac{a}{p})=1$, whence $(p)\in P(J/\QQ)$. 
We have therefore shown that $
P(K/\QQ)$ is contained in $P(J/\QQ)$, up to finitely many exceptional
elements.  It now follows from Bauer's theorem \cite[\S VII.13]{neukirch} 
that $J\subseteq K$, so that $\sqrt{a}\in K$.

\end{document}